\documentclass[11pt]{amsart}
\usepackage{graphicx}
\usepackage[all]{xy}
\usepackage{amscd}
\usepackage{amssymb}
\usepackage{amsmath}
\usepackage{amsthm}
\usepackage{amsfonts}
\usepackage{graphics}
\usepackage{epsfig,psfrag}
\usepackage[english]{babel}
\usepackage{latexsym}
\usepackage{mathrsfs}

\newtheorem{theorem}{Theorem}[section]
\newtheorem{corollary}[theorem]{Corollary}
\newtheorem{lemma}[theorem]{Lemma}

\newtheorem{proposition}[theorem]{Proposition}

\newtheorem{problem}[theorem]{Problem}
\newtheorem{conjecture}[theorem]{Conjecture}

\theoremstyle{definition}

\newcommand{\R}{{\mathbb R}}

\theoremstyle{remark}
\newtheorem{remark}[theorem]{Remark}

\begin{document}
\title[Distributions of points  realizing maximum energies]
{Distributions of points on non-extensible closed curves in $\R^3$ realizing maximum energies}
 
\author{Shiu-Yuen Cheng}
\address{Department of Mathematical Sciences, Tsinghua University, Beijing, 100084, CHINA}
\email{sycheng@mail.tsinghua.edu.cn}
\author{Zhongzi Wang}
\address{Department of Mathematical Sciences, Tsinghua University, Beijing, 100084, CHINA}
\email{wangzz18@mails.tsinghua.edu.cn}



\begin{abstract}  
Let $G_n$ be a non-extensible, flexible closed curve  of length $n$ in the 3-space $\R^3$ with $n$ particles $A_1$,...,$A_n$ 
evenly fixed (according to the arc length of $G_n$) on the curve.  Let $f:(0, \infty)\to \R$ be  an increasing and continuous   function.
Define an energy function 

$$E^f_n(G_n)= \sum_{p< q} f(|A_pA_q|),$$ where $|A_pA_q|$ is the distance between $A_p$ and $A_q$ in $\R^3$. 
We address a natural and interesting problem: 
{\it What is the shape of $G_n$ when $E^f_n(G_n)$ reaches the maximum? } 

In many natural cases, one such case being $f(t) = t^\alpha$ with $0 < \alpha \le 2$, the maximizers are regular $n$-gons and in all cases the maximizers are (possibly degenerate) convex $n$-gons with each edge of length 1.




\end{abstract}
\date{}
\maketitle
{\bf Data availability statement:} Not applicable.

\section{Introduction}
\subsection{Results of the paper} The distributions of points under certain constraints draw the attention of many people.  A motivation for our study is as below:

Let $G_n$ be a non-extensible, flexible closed curve  of length $n$ in the 3-space $\R^3$ with $n$ particles $A_1$,...,$A_n$ 
evenly fixed (according to the arc length of $G_n$) on the curve.  
Let  $\mathcal{G}_n $  be the union of all such $G_n$. Let $f:(0, \infty)\to \R$ be  an increasing and continuous   function.  
Define an energy function 
$$E^f_n(G_n)= \sum_{p< q} f(|A_pA_q|) \qquad (1.1)$$where $|A_pA_q|$ is the distance between $A_p$ and $A_q$ in $\R^3$. Note $|A_i A_{i+1}|\le 1$.

\begin{problem}\label{main}
What is the shape of $G_n$ when the energy $E^\alpha_n$ reaches the maximum? 
\end{problem}

Note $E^f_n(G_n)$ relies only on the positions of particles of 
$G_n$, but the positions of 
those particles are constrained by the  non-extensible curve. 
\begin{figure}[htbp]
\begin{center}
\includegraphics[width=300pt, height=140pt]{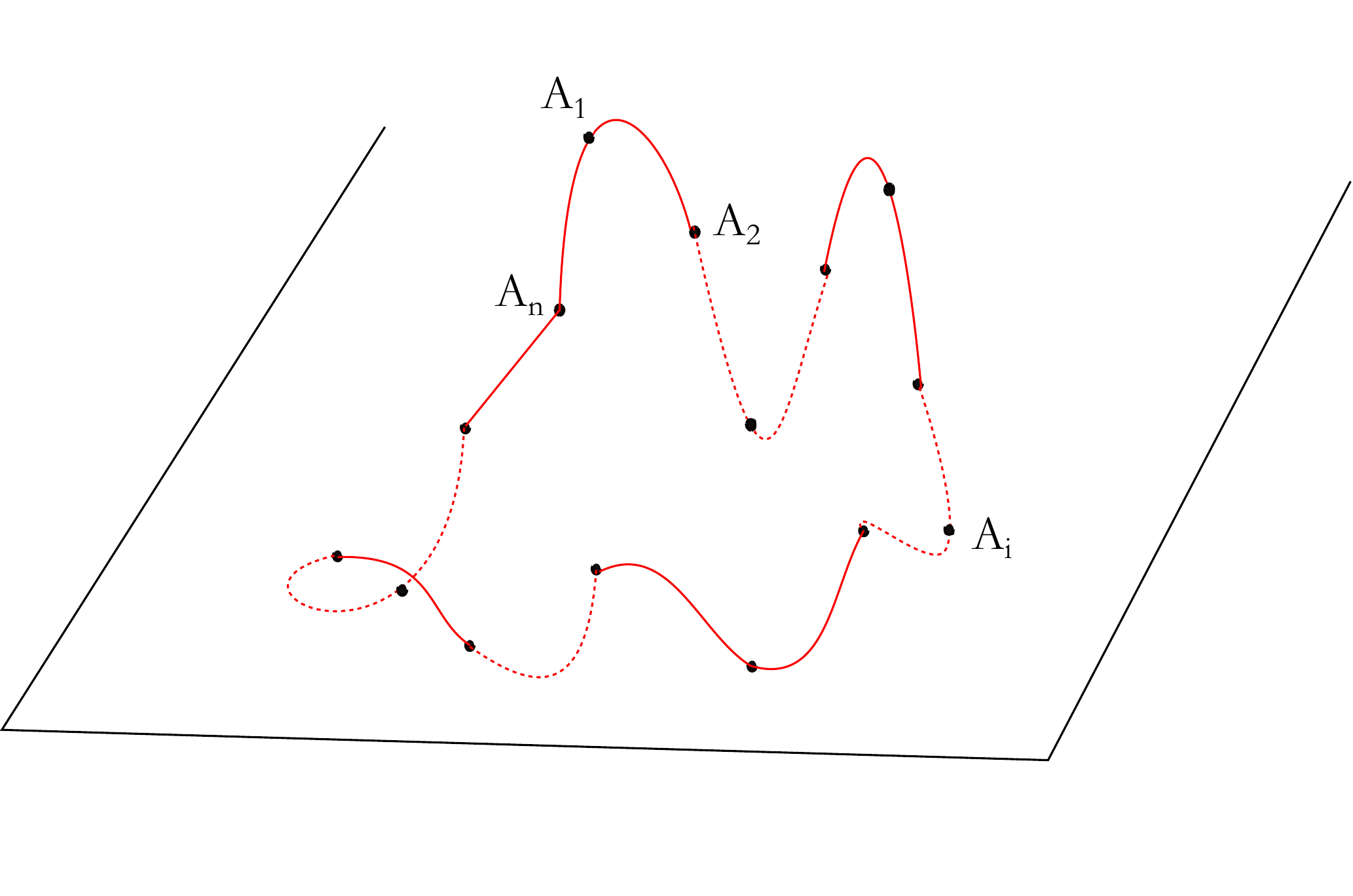}
\end{center}
\centerline{Figure 1}
\end{figure}

Our first two results are about  the existence of  the maximum of $E^f_n$ and where $E^f_n$ reaches the maximum. 

\noindent  {\bf Theorem 2.1.} 
{\it The maximum of $E^f_n$ exists on $\mathcal{G}_n$. }

\noindent  {\bf Theorem 3.1.}
{\it Each maximum point of $E^f_n$ is a convex $n$-gon (possibly degenerate) with each edge of length 1.} 
 
Then we restrict our function $f$ to be  the power functions $f_\alpha,\alpha\in \R $:  
$$f_\alpha(x)=\left\{
\begin{array}{ccc}
x^\alpha, \,\, & \alpha>0;\\
\ln x, \,\, & \alpha=0;\\
-x^\alpha, \,\, & \alpha<0.
\end{array}\qquad (1.2)
\right. $$

For simplicity, below we use  $E^{\alpha}_n$ to denote $E^{f_\alpha}_n$. 
Since  $f_\alpha:(0, \infty)\to \R$ is increasing and continuous, by Theorem 2.1 $E^\alpha_n$ reaches the the maximum on $\mathcal{G}_n$.


Below we use  $ \{\Gamma_n \}$  to denote the set of all convex $n$-gons
with each edge of length 1.   With Theorem 2.1 and Theorem 3.1, we transform Problem \ref{main}  to the following

\begin{problem}\label{main2}
What is the shape of $\Gamma_n$ when the energy $E^\alpha_n$ reaches the maximum? 
\end{problem}

\begin{figure}[htbp]
\begin{center}
\includegraphics[width=280pt, height=80pt]{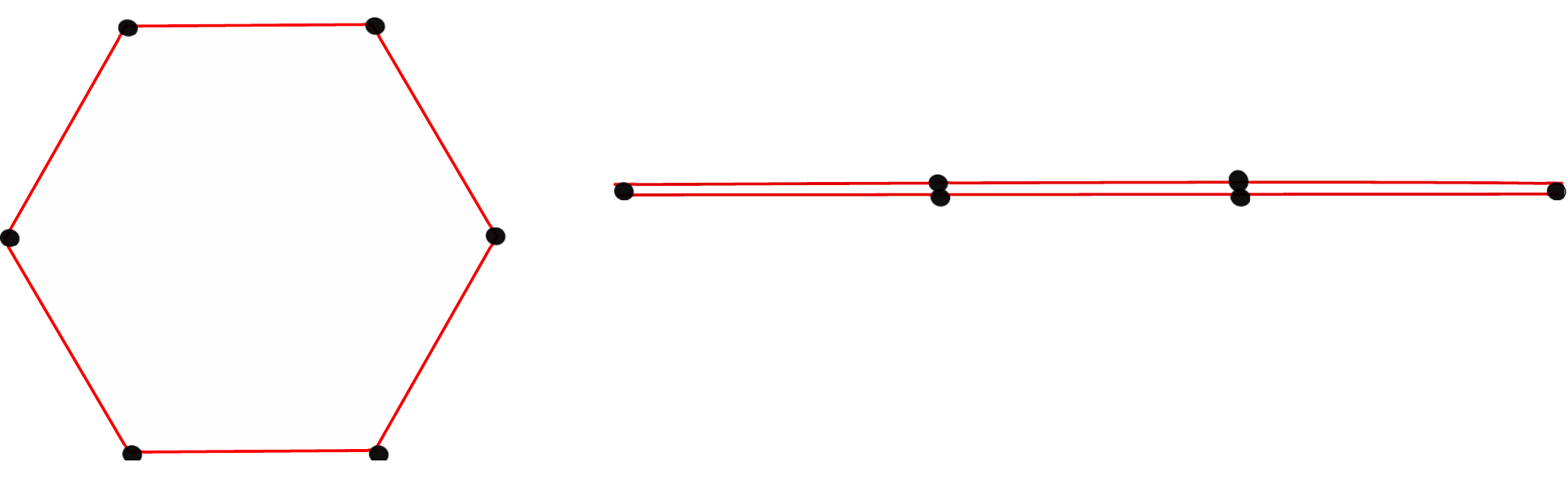}
\end{center}
\centerline{Figure 2}
\end{figure}

There are two extreme shapes for $\Gamma_n$: one is {\bf the regular $n$-gon $\Gamma_n^o$};  the other is {\bf the double straight arc $\Gamma_n^-$}, defined for only $n=2m$ , which can be  defined by
degenerated polygons where one of the diagonals has length $m$ ($\Gamma_{6}^-$ is shown in Figure 2, where two lines coincide indeed). 

It is easy to verify that,  for all $\alpha\in \R$, $E^{\alpha}_2(\Gamma_2)$ reaches the maximum at the double straight arc $\Gamma_2^-$, and
 $E^{\alpha}_3(\Gamma_3)$ reaches the maximum at the regular triangle $\Gamma_3^o$.
The first interesting case is $n=4$ and we have the complete answer: 

\noindent  {\bf Proposition 4.3.} $E^{\alpha}_4(\Gamma_4)$ reaches the maximum at the square $\Gamma_4^o$ for $\alpha<2$ and at the double straight arc $\Gamma_4^-$ for $\alpha>2$,  and $E^{2}_4(\Gamma_4)$ is a constant for all $\Gamma_4$ when $\alpha=2$.

Suggested by the classification in $n=4$, one may expect that for a given $n$,
there will be a constant $\alpha_*$ (respectively $\alpha^*$) so that the shape of $\Gamma_n$ realizing the maximum of $E^\alpha_n$ will be fixed when $\alpha< a_*$  (respectively $\alpha>\alpha^*$). One 
direction of this expectation is true,  and based on a pioneering work of Luko \cite{Luko}, we  will prove

\noindent  {\bf Theorem 5.1}
{For $\alpha\le 2$, $n\ge 5$, $E_n^{\alpha}(\Gamma_n)$ reaches its maximum if and only if $\Gamma_n$ is $\Gamma_n^o$, the regular $n$-gon of edge length 1.}

\begin{conjecture}\label{large enough}
For given even $n>0$, there is a constant  $\alpha_n^*>0$  such that for $\alpha>\alpha_n^*$, $E^\alpha_n(\Gamma_n)$ reaches the maximum if and only if  $\Gamma_n$ is the double straight arc $\Gamma_n^-$.
\end{conjecture}

\subsection{Related results and motivations}

There are many results related to  and prior to  our results.

(1)  
A continuous  version of the problem we considered have been studied over the years:
Given a unit speed curve $c: S^1\to \R^3$ and defining the energy functionals introduced by
O’Hara in \cite{OH}:
$$E_j^p(c)=\int \int \left( \frac 1{|c(s)-c(t)|^j}-\frac 1{d(s,t)^j} \right)^p dt ds$$

  What is the shape of the mimimizers? O’Hara conjectured  the minimizers are round circles when $pj > 2$. Freedman, He, and Wang \cite{FHW} showed that minimizers are convex planar curves and in the special case of the energy $E^2_1$ by showing it is Mobius invariant and used this to show that in this case O’Hara’s conjecture is correct. The O’Hara conjecture was verified in the paper \cite{ACFGH} of Abrams, Cantarella, Fu, Ghomi, and Howard. The current state of understanding about these questions is contained in \cite{ACFGH} and the paper of Exner, Harrell, and Loss \cite{EHL}. In both of these papers it is noted there is ``symmetry breaking" phenomena:  for the energy
  $$E^p(c)=\int \int |c(t)-c(s)|^pdsdt,$$
$E^p$ is maximized by the round circle for $-1<p \le2,$ and for some $p^*\ge 2$ the circle
is no longer the maximizer for $p > p^* $. These papers give some numerical 
calculations estimating and bounding the value of $p^*$. That $E^p$ is maximized by circles for $-1 < p \le 2 $ for planar curves was already in a 1966 paper of Luko \cite{Luko}.

Theorem 3.1, together with Theorem 2.1,  is an analogue of the result of Freedman, He, and Wang mentioned above and which is a meaningful 
 result in its own right. Our proof in the discrete case has more  geometric flavor 
 and somewhat trickier. 

Theorem 5.1 for $n\ge 5$ is an analogue of the result of  Luko mentioned above. In discrete case, we do not have to worry about the convergence
of the integral, so the condition $p>-1$ mentioned above is not needed.

Proposition 4.3 for $n=4$ indicates the  “symmetry breaking''  
mentioned above is happened exactly at $p=2$, and even stronger,
the maximizer suddenly changed from the square to the double straight arc.
Theorem 5.1 and Conjecture 1.3 predicate  the symmetry breaking phenomena for general $n$.
A supporting evidence of Conjecture 1.3 is  
$$\lim_{\alpha\to \infty} E_n^\alpha(\Gamma_n)^{1/\alpha}= \max_{1\le p, q\le n} |A_pA_q|$$
which implies that for large $\alpha$ the maximizers will tend to maximize
the diameter.

As pointed out by the reviewers, some of the results of \cite{Sal} and \cite{CDR} resemble our Theorem 3.1. In the case of \cite{Sal} the setting (discrete vs continuous) is different and in \cite{CDR} they only consider the two dimensional case and do not allow for the double straight arcs $\Gamma_n^-$. As $\Gamma_n^-$ are the maximizers for some choices of the functional $E_n^f$ the results of \cite{CDR} do not directly imply the results here.


(2) Other related problems and papers including:

The Thomson type problem 
considers the distribution of $n$ points on the unit sphere in $\R^3$ under essentially the same energy functions
$f_\alpha$ given by (1.2).  The problem was first 
raised by Thomson for $\alpha=-1$  for his atomic model \cite{Th}, and was later generalized to all $\alpha\in \R$. 
Smale put Thomson's  problem in his 18 problems  for 21st century \cite{Sm}. 
Little is known about Thomson's  problem:
for example, for $\alpha= 1, -1$, the shape realizing the maximum energy  is known only for $n\le 5$. Some related papers are \cite{BG},  
\cite{KuSa}, \cite{KaSh},  \cite{PB},  \cite{HS}, \cite{Sch}. 

 Many people have studied the distribution of $n$ points with mutual  distances $\le 1$ under the  energy functions
$f_\alpha$ given by (1.2) for $\alpha>0$. The case  $\alpha=1$, where the energy function is the sum of the mutual  distances,  was first considered by Toth for points with mutual distances $\le 1$ \cite{To}, and later generalized to all $\alpha>0$. Some related papers are \cite{Wi},  \cite{Pi},  \cite{LP},  \cite{St},  \cite{AGHPMP}.

(3) Some inspiration from the  physics:
The total energy $E^\alpha_n$ we studied has physical meaning for  $\alpha=-1$ and $\alpha
=2$: 
For $\alpha=-1$,
$-E^{-1}_n(\Gamma_n)$ is the total electric potential energy, 
where each vertex of $G_n$ has a unit charge, and there is no charge on the edges.  
For $\alpha=2$, 
$E^{2}_n(\Gamma_n)$ is the moment of inertia  
about its mass center, where each vertex of $\Gamma_n$ has a unit mass, and there is no mass on the edges (see $\S$4.3).
Theorem \ref{main3} for $\alpha=2$ implies that $\Gamma_n$ reaches the maximum moment of inertia about its mass  center at the regular
$n$-gon $\Gamma^o_n$. 
This matches some of our observations: If a dancer spins rapidly, 
or  someone  rotates a necklace quickly with one finger
and then throws it out, then the shapes of the bottom edge of the dress and the necklace will be a regular $n$-gon or a round circle.


{\bf Acknowledgement:} The paper is greatly benefited from the reviewers's advice:

(1)   Subsection 1.2 (1),  the continuous  version of the problem we considered, its current state and connections to our work,  is  written mostly following the reviewers's report.

(2) The reviewers's suggestions make the proof  of Proposition 3.2  (iii) clearer.

(3) Theorem 5.1 was stated as a conjecture in the early version of the paper, and we obtained some partial results of the conjecture. Our approach 
is based on to decompose the total energy $E^\alpha_{n}$ into the sum of $k$-step energies $E^\alpha_{n,k}$ (see the beginning of Section 5) and then apply Jensen inequality to each $k$-step energy. The reviewers  pointed out that such approach has been used by Luko
in 1966. The  reviewers even expected that the conjecture can be derived from Luko's work. Then we have Theorem 5.1.

We thank the  reviewers for their advice.


\section{The existence of the maximum for $E^f_{n}$.}\label{maximum}

\begin{theorem}\label{exist1}  
Let  $f:(0,\infty) \to \R$ be a continues and increasing function. 
The maximum of $E^f_n$ defined by (1.1) exists on $\mathcal{G}_n$. 
\end{theorem}

\begin{proof}
We use $(x_1, x_2, ...,x_n)$, where each $x_i$ is a vector in $\R^3$, to denote the vertices of
 $G_n\subset \R^3$. 
 
 For $i \in \{1,2,...,n\}$ define $B_j$ to be the subset of $(\R^3)^n$
$$B_i :=\{(x_1,x_2,...,x_n)||x_{i+1}-x_i|\le 1\}$$ 
for $i<n$ and $B_n :=\{(x_1,x_2,...,x_n)||x_1-x_n|\le 1\}$.
 

Positions of $(x_1,x_2,...,x_n)$ form a subset $B'\subset (\R^3)^n$ which is defined by
$$B'=\bigcap_{i=1}^n B_i.$$

Since each $B_i$, defined by $\le$,  is closed, their intersection $B'$ is closed.
Since $E_n^f(G_n)$ is invariant under Euclidean transformations, so we may assume that $x_1=0$.
Note $B''\subset (\R^3)^n$ defined by $x_1=0$ is also a closed subset. Let 
$$B=B'\cap B''.$$ 
$B$  is also closed.

To consider the value of $E_n^f$, we need only restrict our attention on $B$.
Since $|x_i-x_1|\le i-1 < n,$
we have $|x_i| < n$,  so
$$d((x_1, x_2, ...., x_n), 0)^2=|x_1|^2+|x_2|^2+...+|x_n|^2\le n^3,$$
where $d$ is the distance of $(\R^3)^d=\R^{3d}$, hence $B$ is bounded.
By Heine-Borel theorem \cite{Ar}, as a closed bounded subset of Euclidean space, $B$ is compact.

Recall $f:(0, \infty)\to \R$ is increasing and continuous. Since $f$ is increasing, as $x\to 0_+$, we have
 either (i) $f(x)\to c$ for some constant $c$, or (ii) $f(x) \to -\infty$.

In case (i),  $f$ is continuous on $[0, \infty)$, hence $E_n^f$ is continuous on the compact set $B$. So $E_n^f$ has a maximum on $B$.

If case (ii), now for $i\ne j$ and some $\epsilon> 0$, let
$$B_{i,j}^\epsilon=\{(x_1,x_2,...,x_n)|\,\,|x_i-x_j|\ge\epsilon\},$$
then $B_{i,j}^\epsilon$ is a closed subset. Let
$$B^\epsilon=B\cap(\bigcap_{i, j}B_{i,j}^\epsilon).$$
As a closed subset of a compact set $B$, $B^\epsilon$ is compact.
For any $(x_1, x_2,..., x_n)\in B^\epsilon$,  $|x_i-x_j| \ge \epsilon$ for any $i \ne j$, so $E_n^f$ is defined on $B^\epsilon$.
By the same reason as before, $E_n^f$ reaches maximum on $B^\epsilon$.

Once the (ordered) vertices of $G_n$ belong to $B^\epsilon$, we  simply  write $G_n\in B^\epsilon$.
Below we assume that $\epsilon <1$.  Then the regular $n$-gon $\Gamma_n^o\in B^\epsilon$.
Recall that $f$ is increasing on $(0,\infty)$, and $f(x) \to -\infty$ as $x\to 0$. First we have $f(|x_i-x_j|)\le f(n-1)$
since $|x_i-x_j|\le n-1$. Denote $f(n-1)$ by $K$. 
Next we can pick $\epsilon>0$ so that $f(\epsilon)<E_n^f(\Gamma_n^o)-(C_n^2-1) K$. 
If $G_n\notin B^\epsilon $, then $|x_k-x_m|<\epsilon$ for some $k \ne m$. Therefore 
$$E_n^f(G_n)=\sum_{i<j}f(|x_i-x_j|)\le f(|x_k-x_m|) +\sum_{i<j, (i,j)\ne (k ,m)}f(|x_i-x_j|) $$$$<f(\epsilon)+(C_n^2-1)K<E_n^f(\Gamma_n^o).$$
So the value of $E_n^f$ on $B\setminus  B^\epsilon$ is bounded by $E_n^f(\Gamma_n^o)$. Since $\Gamma_n^o\in B^\epsilon$, the maximum value of  $E_n^f$ on $B$ is the maximum value on $B^\epsilon$.
So  the maximum value of $E_n^f$ exists.
\end{proof}

Since $f_\alpha$ is continuous and increasing on $(0,\infty)$ for all $\alpha\in \R$, by Theorem 2.1 we have

\begin{corollary}\label{exist}
For each $\alpha$ and $n$,   $E^\alpha_n$ reaches the maximum on $\mathcal{G}_n$.
\end{corollary}

\section{Each maximum point of $E^f_n$ is a convex $n$-gon (which may be degenerate) with each edge of length 1.}\label{shape}

A subset $X\subset \R^3$ is convex if it contains the line segments connecting each pair of its points. 
The convex hull of $X$ is the (unique) minimal convex set containing $X$.
Suppose $S$ is a set of finitely many points.
The boundary of the convex hull of  $S$ forms a convex polytope and forms a convex polygon if  $S\subset \R^2$.

When the points of $S$
are in a line in $\R^2$,  Let $P_1$ and $P_2$ be the outermost two points. We will consider the convex hull of $S$  as a degenerated convex polygon, who's boundary is still a closed curve which
consists of  two coincided straight arcs connecting $P_1$ and $P_2$,  and the exterior angles at $P_1$ and $P_2$
are $\pi$, see the right of Figure 3. 

\begin{figure}[htbp]
\begin{center}
\includegraphics[width=250pt, height=80pt]{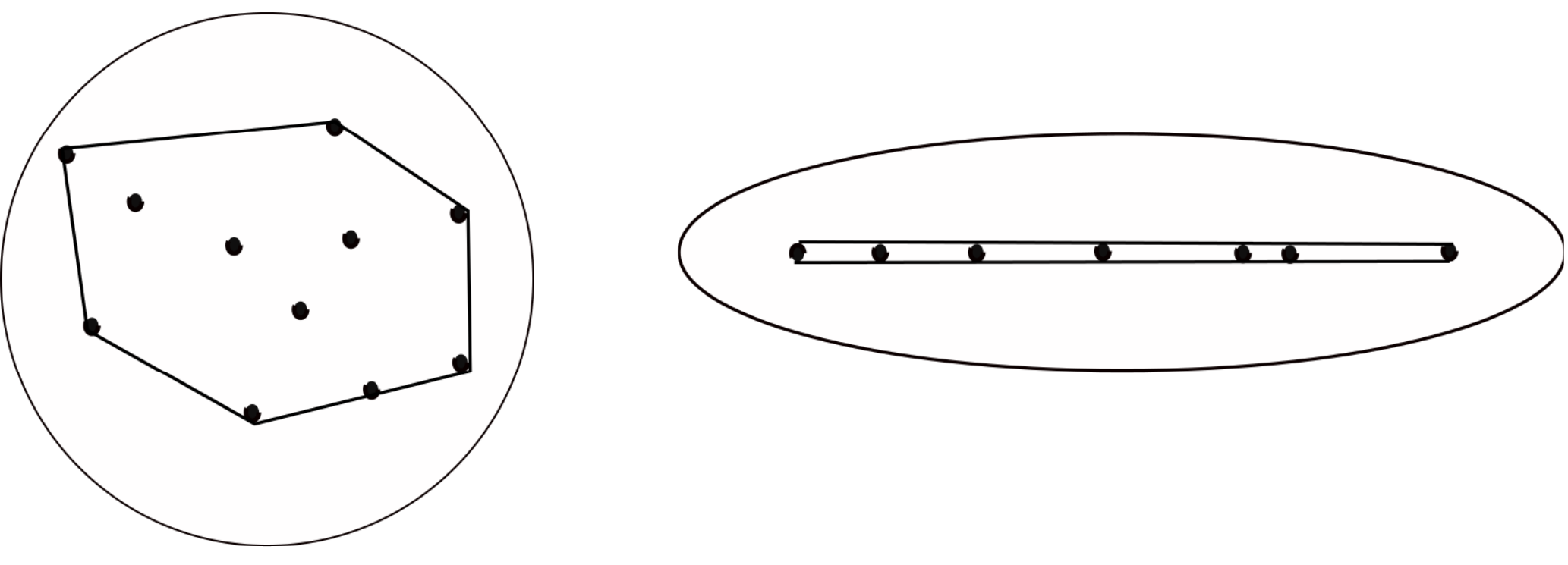}
\end{center}
\centerline{Figure 3}
\end{figure}

\begin{theorem}\label{convex1}
Each maximum point of $E^f_n$ is a convex $n$-gon (which may be degenerate) with each edge of length 1.
\end{theorem}

Recall {\bf the vertices $A_1,..., A_n$ are cyclicly consecutive in $G_n$}. 
Theorem \ref{convex1} follows from the following proposition whose statement gives the steps of the proof.

\begin{proposition}\label{convex2}
Suppose $E^f_n$ reaches the maximum at  $G_n$. Then 

(i) All vertices of $G_n$ are in the same plane.

(ii) Let $C$ be the 2-dimensional convex hull of all vertices of $G_n$, then all vertices of  $G_n$ stay in $\partial C$, the boundary of $C$. 

(iii) Suppose $A_i$ and $A_{i+k}$ are two vertices of $G_n$ in an edge $L$ of $\partial C$ such that there are no vertices of $G_n$ between $A_i$ and $A_{i+k}$ in $L$, then the distance between $A_i$ and $A_{i+k}$ is no more than $1$.


(iv)  $G_n$ is a convex $n$-gon of edge length 1.
\end{proposition}


In the conclusion of (ii), the cyclic order of vertices 
in $\partial C$ usually is not the same as that in $G_n$, see Figure 4.
Also may be $A_i=A_j$ on $\partial C$, and $C$ can be degenerated.

\begin{figure}[htbp]
\begin{center}
\includegraphics[width=240pt, height=100pt]{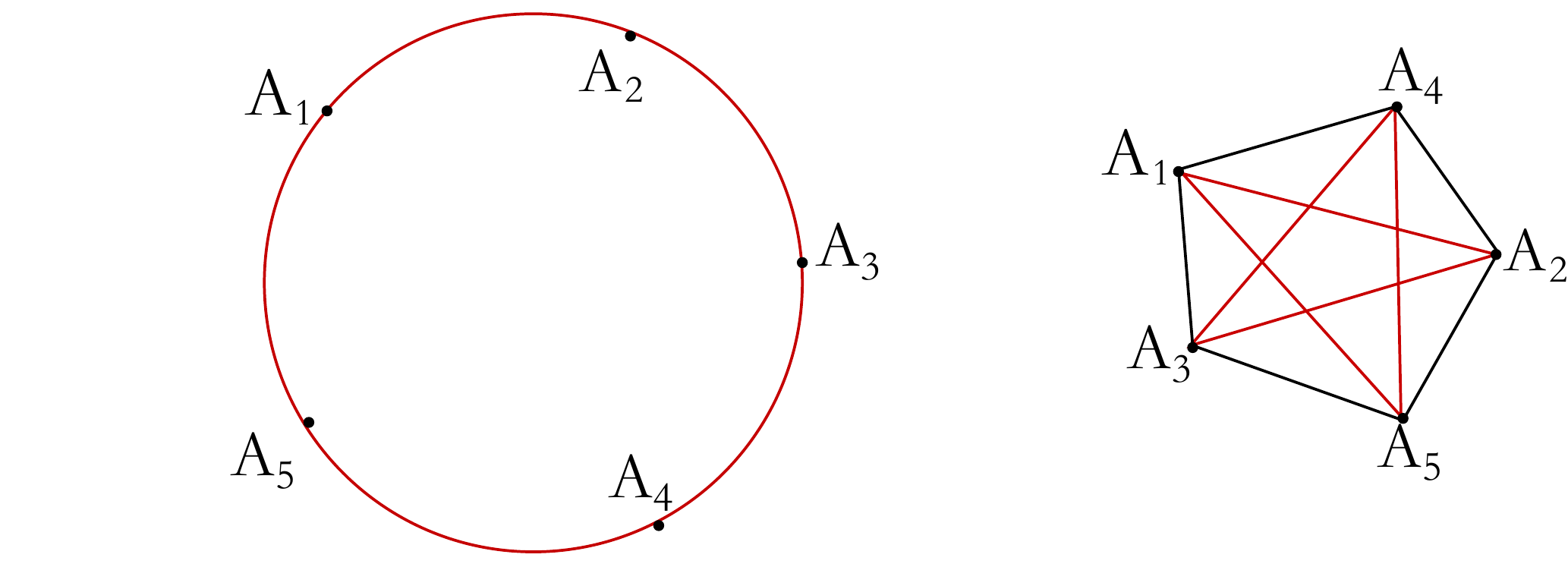}
\end{center}
\centerline{Figure 4}
\end{figure}

Then the statement (iii) makes sense under the following 

Conventions(*):  When  $C$ is degenerated,  then $\partial C$
consists of  two coincided straight arcs $C_1$ and $C_2$, and $\partial C$ 
 travels first along  $C_1$ then along $C_2$. 
Pick any vertex  $A_i$ at one end and another vertex $A_{i+k}$ at  the other end. 
We always assume that all vertices from $A_i$ to $A_{i+k}$ in $G_n$ stay in $C_1$ and remaining  vertices (from $A_{i+k+1}$ to $A_{i-1}$ in $G_n$) stay in $C_2$. 

\begin{proof}
(i) Let $\bar C$ be the convex hull of those vertices of $G_n$ (the edges of $G_n$ usually are not in $\bar C$). 
If  those vertices are not contained in any plane,
then $\bar C$ is a 3-dimensional polyhedron, and we pick a face of $\bar C$ and denote the plane containing this face by $\Pi$.

Denoted the vertices in $\Pi$ by $P_1$, $P_2$, ... , $P_k$ according to their orders in the curve $G_n$.
Note all remaining vertices are in one side of $\Pi$.

\begin{figure}[htbp]
\begin{center}
\includegraphics[width=220pt, height=140pt]{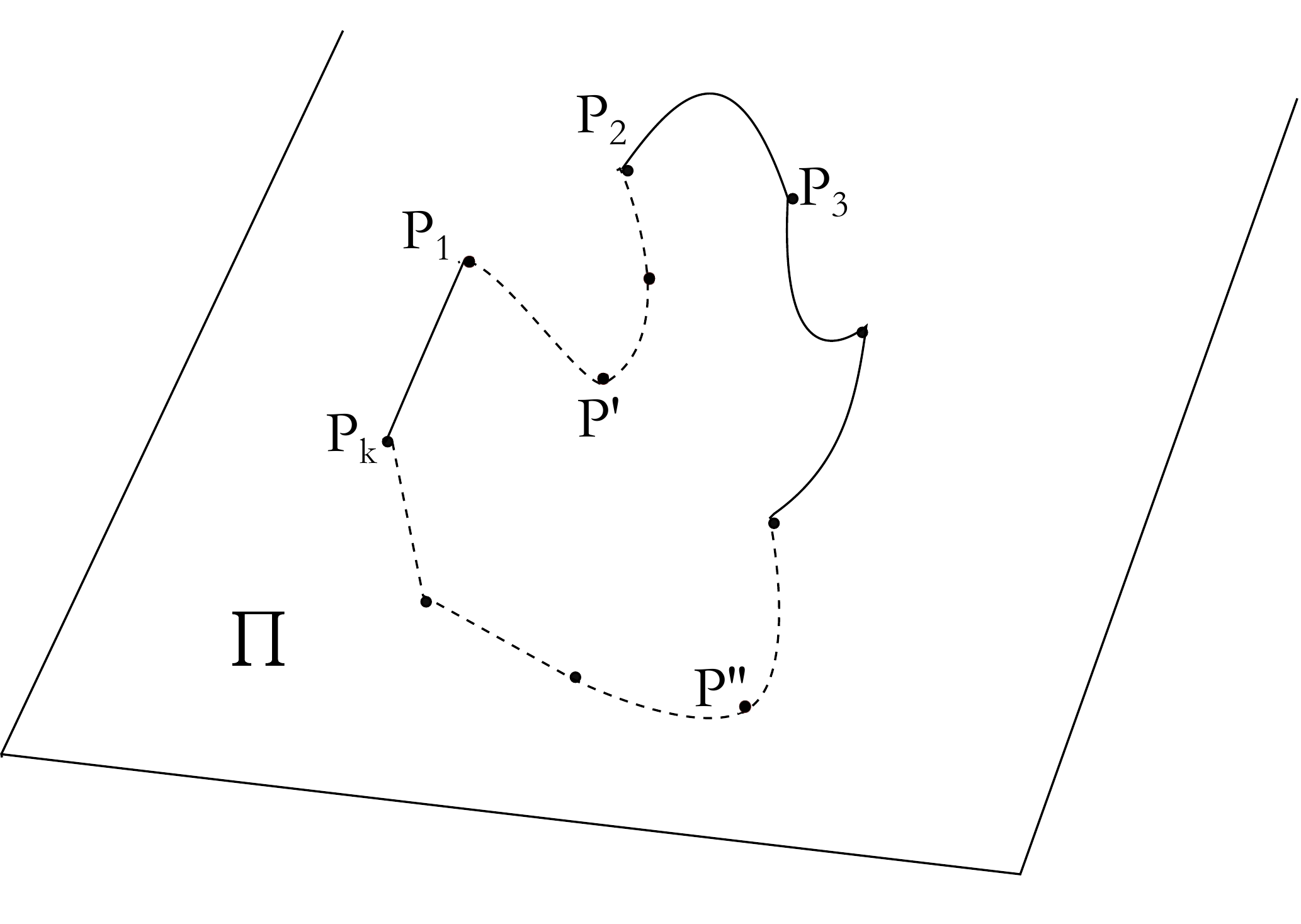}
\end{center}
\centerline{Figure 5}
\end{figure}

If $P_1$, $P_2$,..., $P_k$ are not consecutive in $G_n$, we may assume that $P_1$, $P_2$ are not consecutive in $G_n$. Then 
 $P_1$ and $P_2$ divide $G_n$ into two parts $G'$ and $G''$, each part contains some vertices not in $\Pi$, see Figure 5. Now reflecting
 $G'$ about $\Pi$ we get a new distrubution of $G^*_n\in \mathcal{G}_n$. To compare with the old distribution, the distance  $|P' P''|$ increases
 for each vertex $P'$ of $G'$ and $P''$ of $G''$ who are not in $\Pi$; and the distance of any remaining two vertices are not changed.
 So for the new distribution $E^f_n(G^*_n)$ is larger.
   
 Suppose now  $P_1$, $P_2$, ..., $P_k$ are consecutive in $G$.
 Let $C$ be the convex hull of $P_1$, $P_2$,..., $P_k$ in $\Pi$. Then $\partial C$, the boundary of $ C$, is a non-degenerated convex polygon in $\Pi$.  There are  two
 vertices, say $P_i$ and $P_j$,   consecutive in $\partial C$ but not consecutive in $G_n$ 
 (otherwise all vertices of $G_n$ are already in $\Pi$). 
Then we can rotate $\Pi$
 along the line $L$ passing $P_{i}$ and $P_j$  a very small angle so that except $P_{i}$ and $P_j$, 
 all vertices of $G_n$ are below $\Pi$ (note $G_n$ is invariant when we rotate $\Pi$), see Figure 6. $P_i$ and $P_j$ divide $G_n$ into two parts $G'$ and $G''$, each part contains some points not in $\Pi$. Now we can repeat the same argument in the last paragraph to show $E^f_n(G_n)$ can not be the maximum.
We have proved that all vertices of $G_n$ are in the same plane when $E^f_n(G_n)$ reaches the maximum.

 \begin{figure}[htbp]
\begin{center}
\includegraphics[width=220pt, height=140pt]{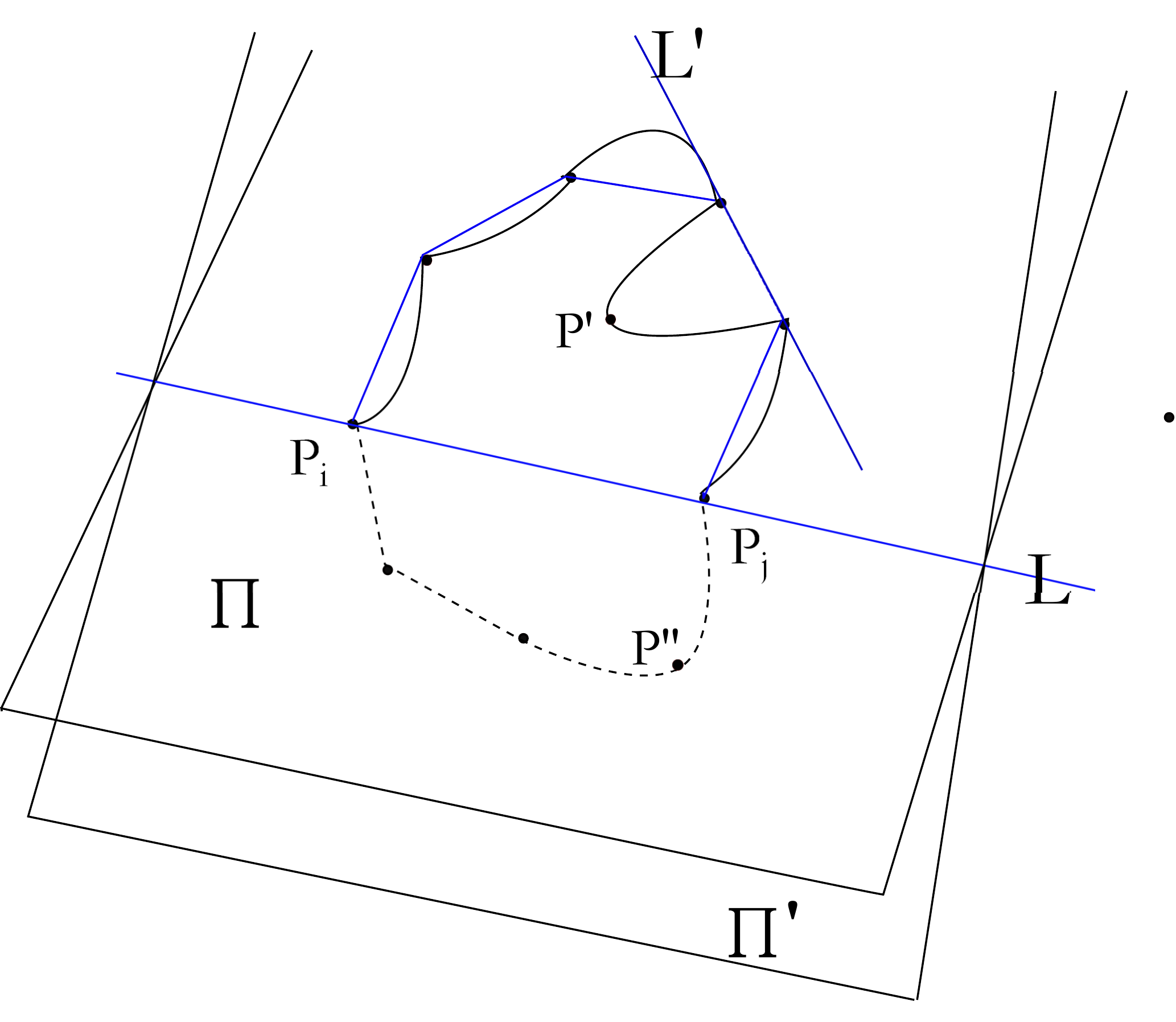}
\end{center}
\centerline{Figure 6}
\end{figure}

(ii) By (i), we assume now all vertices of $G_n$ are in the plane $\Pi$.
Suppose some vertex $P'$ of $G_n$ is in the interior of $C$ (still refer to Figure 6). 
Then again some line $L'$ in $\Pi$ (see Figure 6) contains an edge of $ C$ which
  divides $G_n$ into two parts $G'$ and $G''$, each part contains some points not in $L'$.
  Since the positions of edges of $G_n$ do not affect $E^f_n(G_n)$, for convenience, we may assume that $G_n$ is in $\Pi$. 
  Reflecting $G'$ about $L'$,
we can repeat the same argument as in (i) to show $E^f_n(G_n)$ can not be the maximum.   
We have proved that all vertices of $G_n$ are in $\partial C$. 

Below we will still use $A_1,..., A_n$ to replace  $P_1,...., P_k$.

(iii) Suppose $L$ is an edge of the polygon $\partial  C$. 

\begin{figure}[htbp]
\begin{center}
\includegraphics[width=300pt, height=160pt]{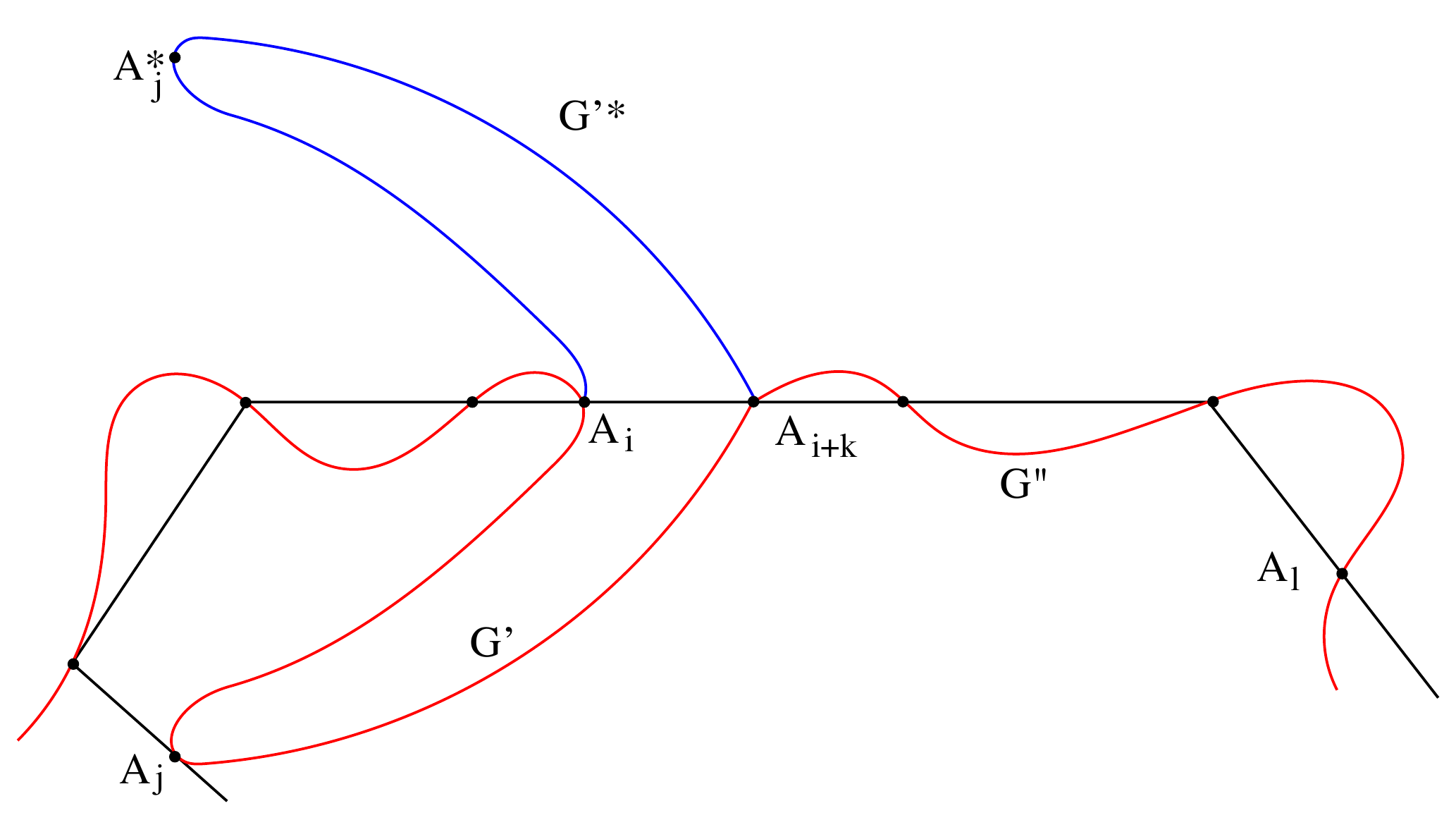}
\end{center}
\centerline{Figure 7}
\end{figure}

Claim (a):  
For any two vertices $A_i$ and $A_{i+k}$  in $L$,  either $\{A_{i+1}, ... , A_{i+k-1}\}$ 
must be in $L$, or $\{A_{i+k+1},...., A_{i-1}\}$ must be in $L$, where $\{A_{i+1},..., A_{i+k-1}\}$ denotes  the set of all vertices from  $A_{i+1}$ to $A_{i+k-1}$  in $G_n$,
and $\{A_{i+k+1},..., A_{i-1}\}$ has similar meaning.

Proof of Claim (a): If $C$ is degenerated, this follows from the  convention (*). 
Now  suppose  $C$ is non-degenerated and   $\partial C$ is a convex polygon as in Figure 7.
If Claim (a) is not true,  then we have some $A_j\in \{A_{i+1}, ..., A_{i+k-1}\}$ and $A_l\in  \{A_{i+k+1},...., A_{i-1}\}$, both $A_j$ and $A_l$ are not in $L$.
Then  $A_i$ and $A_{i+k}$ divide $G_n$ into two parts $G'$ and $G''$ with $A_j\in G'$ and $A_l\in G''$. Both of $A_j$ and $A_l$ must be in the same side of $L$.  
Now reflecting $G'$ about the line containing $L$ to get $G^{'*}$, and let $A_j^*\in G^{'*}$ be the image of $A_j$,
we have a new distribution $G^*_n=G^{'*}\cup G''\in \mathcal{G}_n$.
Clearly $|A_j^*A_l|>|A_jA_l|$ and  the distance of any remaining two vertices are not decreased.
 So for the new distribution $E^f_n(G^*_n)$ is larger, a contradiction

We are going to prove (iii): Note each vertex of the convex polygon $\partial C$ must be  a vertex of $G_n$.    So the vertices $A_i$ and $A_{i+k}$ satisfying the assumption of (iii) must be lying on an edge  $L$ of $\partial C$.
We may suppose $L$ is in horizontal position and $A_i$ is on the left of $A_{i+k}$, see Figure 7 (also Figure 7+).
By Claim (a), either $\{A_{i+1}, ..., A_{i+k-1}\}$, or $\{A_{i+k+1},...., A_{i-1}\}$ must be in $L$. 
We may assume that $\{A_{i+1}, ..., A_{i+k-1}\}$ are in $L$.

Let $j$ be the minimal integer such that $A_{i+j}$ is not on the left side of 
$A_{i+k}$, $j=1,..., k$. Then $A_{i+j-1}$ must be on left side of $A_{i+k}$.  Since there is no vertex between $A_i$ and $A_{i+k}$, $A_{i+j-1}$ is not on the right side of $A_i$. This implies that the interval $A_{i+j-1}A_{i+j}$
contains the interval $A_iA_{i+k}$, see Figure 7+. Since  $|A_{i+j-1}A_{i+j}|\le 1$, we have $|A_i A_{i+k}|\le 1$. (In the above argument, it is possible that
$A_{i+j}=A_{i+k}$ and $A_{i+j-1}=A_{i}$).

\begin{figure}[htbp]
\begin{center}
\includegraphics[width=260pt, height=50pt]{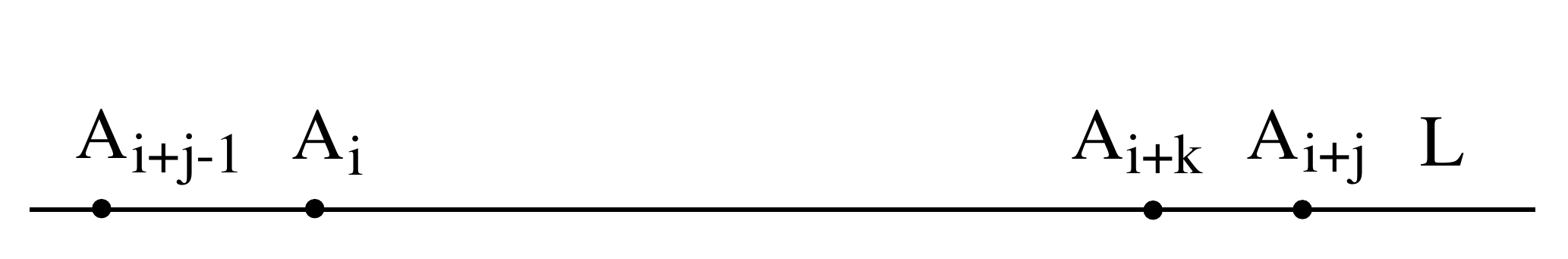}
\end{center}
\centerline{Figure 7+}
\end{figure}

(iv) Suppose  the vertices of $G_n$ appear  in $\partial C$ consecutively as  
$Q_1, Q_2, ..., Q_l$ with multiplicity $q_1, q_2, ..., q_l$, $\sum_{i=1}^l q_i=n$ (recall in the degenerated case, travel first in $C_1$
and then in $C_2$).
By (iii), $|Q_i Q_{i+1}|\le 1$.
Then there is  $ \Gamma_n\in \mathcal{G}_n$ which sends the first $q_1$ vertices  $A_1$,..., $A_{q_1}$ to $Q_1$, the next $q_2$ 
vertices $A_{q_1+1},..., A_{q_1+q_2}$ to $Q_2$,..., and the last $q_l$ vertices to $Q_l$.  
Now the vertices of in $\partial C$ are in the cyclic order $A_1,$ $A_2$,..., $A_n$ (when $A_i=A_{i+1}$ we read $A_i$ first, then $A_{i+1}$).  Clearly $E^f_n(G_n)=E^f_n(\Gamma_n)$.

\begin{figure}[htbp]
\begin{center}
\includegraphics[width=260pt, height=130pt]{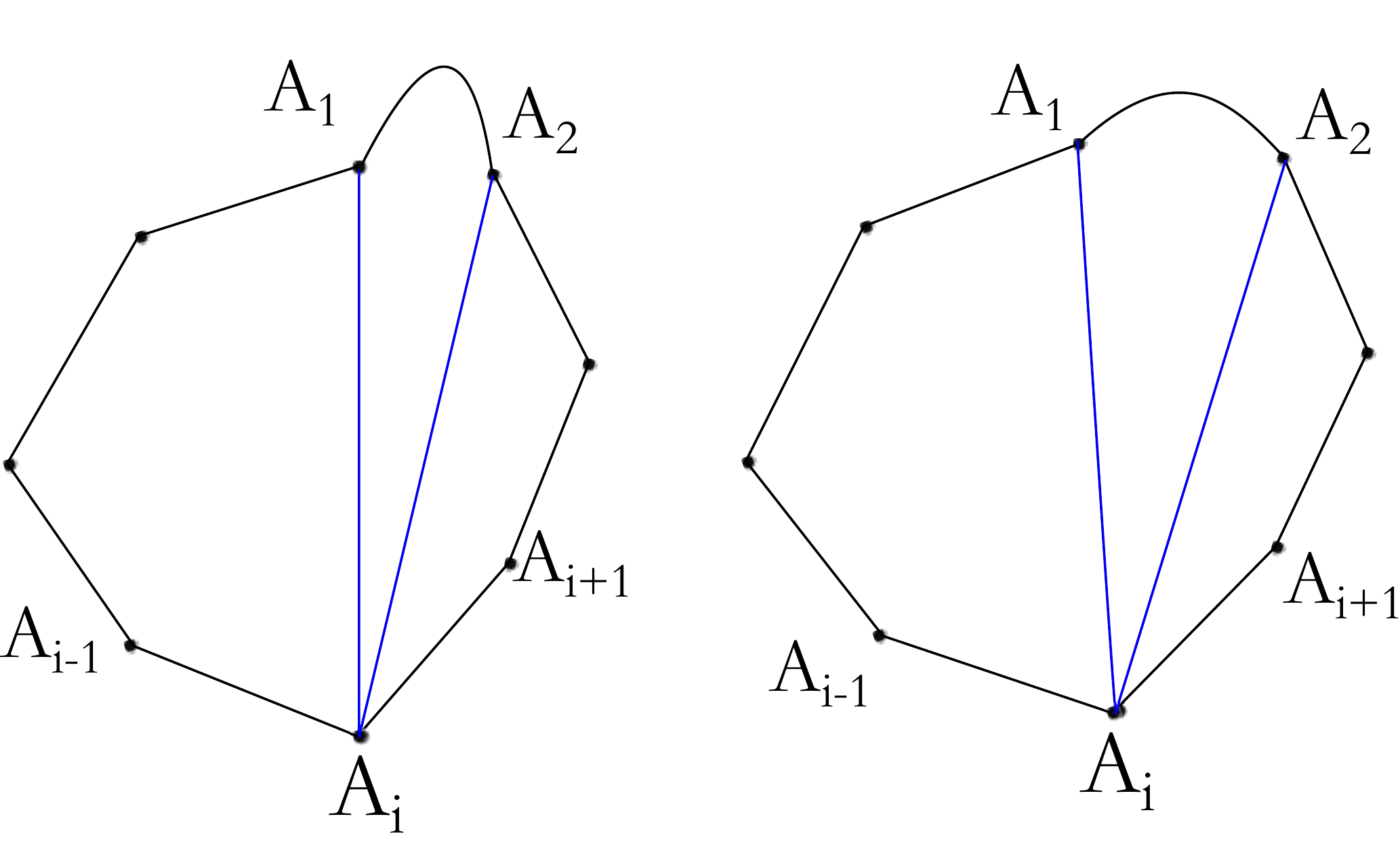}
\end{center}
\centerline{Figure 8}
\end{figure}

Suppose $G_n$ is not a convex $n$-gon of edge length one, then the distance of two consecutive vertices in $\Gamma_n$, say $A_1$ and $A_2$,  is less than 1, that is the unique edge $e$ of $\Gamma_n$ connecting
$A_1$ and $A_2$ is not straight. 
  
Let  $A_i$ be the vertex such that $|A_i A_1|$ is maximum. Then the angle $angle A_{i-1}A_iA_{i+1}$  must be less than $\pi$ (otherwise contradicts that $|A_i A_1|$ is maximum). Now $e$ and $A_i$ divide $\Gamma_n $
   into two parts $G'$ and $G''$, $G'$ contains $A_1$ and $G''$ contains $A_2$.
  Let $G^1$ be the union of $G'$ and the segment  $A_1A_i$ and  $G^2$ be the union of $G''$ and the segment  $A_2A_i$.
  Now keep both $G^1$ and $G^2$ rigid. Then rotate slightly $G^2$
around $A_i$ to increase the angle $angle A_{i-1}A_iA_{i+1}$ slightly but still less than $\pi$. 
We can do this since the unique edge $e$ connecting $G^1$ and $G^2$ is not straight.  Since each $G^1$ and $G^2$ are rigid, and 
the angle $angle A_{i-1}A_iA_{i+1}$ is increasing but still less than $\pi$,
it is easy to see the distances for points in $G^1$ are not changed, the distances for points in $G^2$ are not changed, but  for each $A_k$ in $G^1$, $A_l$ in $G^2$, $k,l \ne i$, the distance
$|A_k A_l|$ is increasing by using cosine theorem. So $E^f_n(\Gamma_n)=E^f_n(G_n)$ can not be the maximum, which contradicts
our assumeption. We have proved (iv), that is $G_n$ is a convex $n$-gon of edge length 1. \end{proof}

\section{Some basic facts, the classification for $n=4$.}

\subsection{Some primary facts.}  

From now on, for each $\Gamma_n\subset \R^2$, we often consider its vertices $A_1,..., A_n$ as vectors in $\R^2$ and 
to denote  the  vector $A_k-A_i$ by $A_iA_k$. We can talk the addition and inner product of those vectors.

Two classical inequalities below can be found in \cite{HLP}.

\begin{lemma}\label{convex}
(1) (Jensen inequality) Suppose $f$ is a concave function ($f''<0$) on $[a,b]$, $\theta_i\in [a,b ]$.
Then $$\frac{\sum _{i=1}^{n}f(\theta_i)}n\le  f(\frac{\sum_{i=1}^{n}\theta_i}n),$$ and the equality holds if and only if $\theta_1=\theta_2=...=\theta_n$.

(2) (Karamata inequality) 
Suppose $f$ is a convex function ($f''>0$) on $[a,b]$ and there are $n$ variables $x_1,x_2,...,x_n\in [a,b]$ with a fixed sum.
Then the value $\sum _{i=1}^{n}f(x_i)$ reaches the maximum if and only if at least $n-1$ variables are at  endpoints.
\end{lemma}

\begin{lemma}\label{concon}
$f_\alpha(x)$ is an increasing function; furthermore $f_a(x)$ is concave when $\alpha<1$ and is convex when $a>1$.
\end{lemma} 

\begin{proof}
A direct calculation show $f_\alpha'$ is always positive, hence $f_\alpha$ is an increasing function.
Moreover $f_\alpha''$ is negative when $\alpha<1$, hence $f_\alpha$ is concave when $\alpha<1$. 
$f_\alpha''$ is positive  when $\alpha>1$, hence $f_\alpha$ is convex when $\alpha>1$. \end{proof}

\subsection{Classification of when $\Gamma_4$ realizing max$E_4^\alpha$.}
\begin{proposition}\label{n=4} $E^{\alpha}_4$ reaches the maximum at the square $\Gamma_4^o$ for $\alpha<2$ and at the double straight line $\Gamma_4^-$ for $\alpha>2$,  and $E^{2}_4(\Gamma_4)$ is a constant for all $\Gamma_4$.
\end{proposition}

\begin{figure}[htbp]
\begin{center}
\includegraphics[width=300pt, height=100pt]{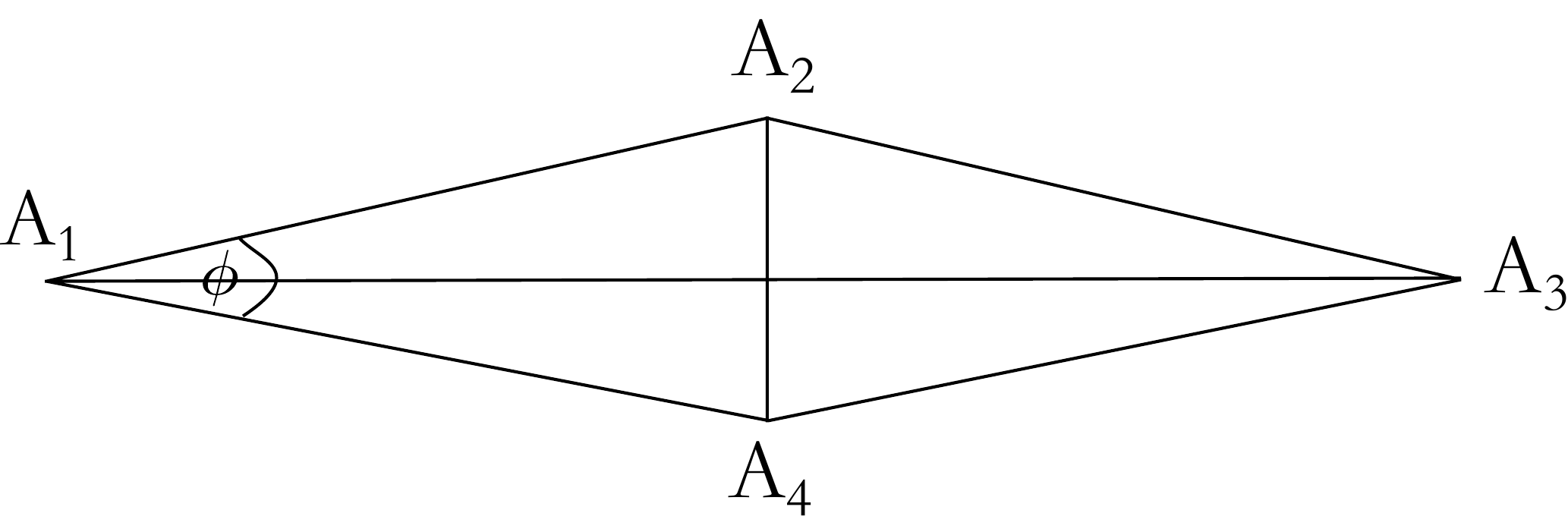}
\end{center}
\centerline{Figure 9}
\end{figure}

\begin{proof} See Figure 9 for $\Gamma_4$. Note in
$$E^\alpha _4(\Gamma_4)=f_\alpha(|A_1A_2|)+f_\alpha(|A_2A_3|)+f_\alpha(|A_3A_4|)+f_\alpha(|A_4A_1|)+f_\alpha(|A_1A_3|)+f_\alpha(|A_2A_4|)$$
the sum of first four terms  is a constant  for any given $\alpha$.
So we need only to classify when $\Gamma_4$ realizing maximum of $f_\alpha(|A_1A_3|)+f_\alpha(|A_2A_4|)$.

Let the inner angle at $A_1$ be $\phi$. Then $\Gamma_4$ is determined by $\phi$. Denote  $f_\alpha(|A_1A_3|)+f_\alpha(|A_2A_4|)$ by $E^\alpha(\phi)$.
We have 
$$E^\alpha(\phi)=f_\alpha(|A_1A_3|)+f_\alpha(|A_2A_4|)=f_\alpha(2\cos  \phi/2)) + f_\alpha(2\sin  \phi/2))$$
$$=f_\alpha(2(\cos^2  \phi/2))^{1/2}) + f_\alpha(2(\sin^2  \phi/2))^{1/2})=
\left\{
\begin{array}{ccc}
2^\alpha(t^{\alpha/2}+(1-t)^{\alpha/2}), \,\, &\alpha>0;\\
\ln 4 + \frac 12 (\ln t +\ln (1-t )), \,\, &\alpha=0;\\
-2^\alpha(t^{\alpha/2}+(1-t)^{\alpha/2}), \,\, &\alpha<0;
\end{array}
\right. $$
where $t=\cos^2 \phi/2$.

By Lemma \ref{concon} $f_\alpha$ is a concave function if $\alpha<1$ and a convex function if $\alpha>1$.

If $\alpha<2$, then $\alpha/2 <1$, we can apply by Jensen inequality to get that  $E^\alpha (\phi)$ reached the maximum if and only if $t=1/2$,
that is $\cos^2 \phi/2=1/2$, that is $\phi=\pi/2$ and therefore $E^\alpha (\phi)$ reaches the maximum 
if and only if $\Gamma_4=\Gamma_4^o$.  

If $\alpha> 2$, then $\alpha/2 >1$, we can apply  Karamata inequality to get $E^\alpha (\phi)$ reached the maximum if and only if $t=0$ or 1,
that is $\cos^2 \phi/2=0$ or 1, that is $\phi=0$ or $\pi$, and  therefore $E^\alpha (\phi)$ reaches the maximum 
if and only if $\Gamma_4=\Gamma_4^-$.  

When $\alpha=2$, then $\alpha/2 =1$, and $ E^\alpha _4(\Gamma_4)$ is a constant 8.\end{proof}

\subsection{$E_n^2(\Gamma_n)$ and moment of inertia.}
\begin{remark}\label{moment of inertia}
If we consider each vertex $A_i$ of $\Gamma_n$ has unit mass, and there no mass on the curve $\Gamma$.  Then $E_n^2(\Gamma_n)$ is the  
the moment of inertia of $\Gamma_n$ about its mass center, up to a constant $n$.
\end{remark}

\begin{proof}
We choose  the mass center
of $\Gamma_n$ be the origin $O$. Then by definition
$\sum_{i=1}^n A_i=0$.
Now 

$$E_n^2(\Gamma_n)=\sum_{i<j} |A_i-A_j|^2=\sum_{i<j} \left< A_i-A_j, A_i-A_j\right> $$
$$=\frac 12 \sum_{i, j} \left< A_i-A_j, A_i-A_j\right> $$
$$=\frac 12 \sum_{i, j} (\left< A_i, A_i\right>-\left< A_i,A_j\right>-\left< A_i, A_j\right>+\left< A_j, A_j\right>)$$
$$=\frac 12 \sum_{i, j} (\left< A_i, A_i\right>-2\left<A_i, A_j\right>+\left< A_j, A_j\right>)$$
$$=n(\sum_{i}  |A_i|^2- \sum_{i, j} \left< A_i, A_j\right>)$$

On the other hand 

$$\sum_{i, j} \left< A_i, A_j\right>=\left<\sum_{i=1}^n A_i, \sum_{i=1}^n A_i\right>=\left<0,0\right>=0$$

So we have  $$E_n^2(\Gamma_n)=n\sum_{i}  |A_i|^2.$$
That is to say, $E_n^2(\Gamma_n)$ is the  
the moment of inertia of $\Gamma_n$ about its mass center, up to a constant $n$.

\end{proof}

\section{When the regular $n$-gon $\Gamma_n^o$ realizes  max$E^{\alpha}_n$}

\begin{theorem}\label{main3}
For $\alpha\le 2$, $n\ge 5$, $E_n^{\alpha}(\Gamma_n)$ reaches its maximum if and only if $\Gamma_n$ is $\Gamma_n^o$, the regular $n$-gon of edge length 1.
\end{theorem}

For any $x>0$, let $[x]$ be the maximum integer not bigger than $x$.

Let $\Gamma_n$ be a convex $n$-gon with each edge of length 1 in $\R^2$ with vertices $A_1,...., A_n$. 
The following decomposition of  $E^{\alpha}_n$ is an important step in the proof of Theorem \ref{main3}.
$$E^\alpha_{n}(\Gamma_n)=\sum_{k=1}^{[n/2]} \mu_{n,k}E^\alpha_{n,k}(\Gamma_n), \qquad (5.1)$$
where  
$$E^\alpha_{n,k}(\Gamma_n)= \sum_{i=1}^n f_\alpha (|A_iA_{i+k}|) \qquad (5.2)$$
where $\mu_{n,k}=1/2$ if $n$ is even and $k=n/2$ and $=1$  for the remaining cases. 

In Figure 10, the interactions of $E^\alpha_{n,k}$ along the 
black lines for $(n,k)=(6,1), (7,1)$, along the blue lines for $(n,k)=(6,2), (7,2)$, and along the red lines for $(n,k)=(6,3), (7,3)$.

\begin{figure}[htbp]
\begin{center}
\includegraphics[width=240pt, height=120pt]{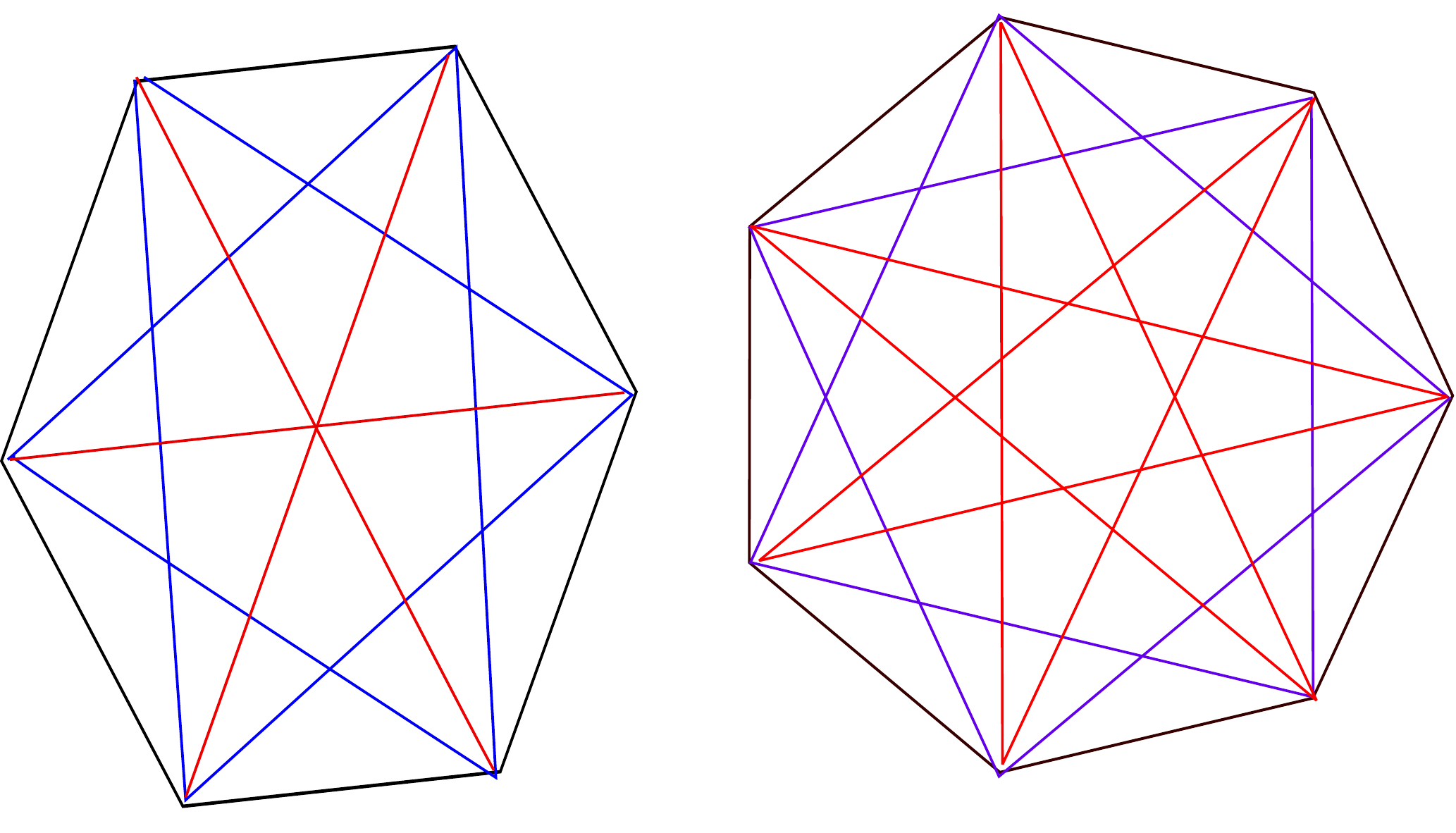}
\end{center}
\centerline{Figure 10}
\end{figure}

The proof of Theorem \ref{main3} is based on the following two theorems of Luko.

\begin{theorem}\label{1}
(\cite[Theorem II]{Luko})
Let $\Gamma_n$ be a convex $n$-gon with each edge of length $1$ and let $g:(0,\infty)\rightarrow \mathbb{R}$ be an increasing concave function.  Then for $n\ge 4$, $1 \le k \le [n/2]$, $k$ is fixed, the inequality
$$\frac{1}{n}\sum_{i=1}^{n}{g(|A_iA_{i+k}|^2)\le g(\sin^2\frac{k\pi}{n}/\sin^2\frac{\pi}{n})}, $$
holds. The sign of equality holds if and only if $\Gamma_n=\Gamma_n^{o}$.
\end{theorem}

\begin{theorem}\label{2}
(\cite[Theorem III]{Luko})
Let $\Gamma_n$ be a convex $n$-gon with each edge of length $1$.   Then for $n\ge 4$, $1 \le k \le [n/2]$, $k$ is fixed, the inequality
$$\sum_{i=1}^{n}|A_i-A_{i+k}|^2\le(\sin\frac{k\pi}{n}/\sin\frac{\pi}{n})^2\cdot n$$
holds.
The equality holds if and only if $\Gamma_n=\tau(\Gamma_{n}^{o})$, where $\tau$ is an affine transformation on $\mathbb{R}^2$.
\end{theorem}

We also need the following fact to prove Theorem \ref{main3}.

\begin{lemma}\label{3}
For $n\ge 5$. Suppose $\Gamma_n$ is a convex $n$-gon with each edge of length $1$ and 
$\Gamma_n=\tau(\Gamma_{n}^{o})$,  where $\tau$ is an affine transformation on $\mathbb{R}^2$. 
Then $\Gamma_n=\Gamma_{n}^{o}$.
\end{lemma}

\begin{proof} We may assume  $\tau(v)=Av+b$ with $A\in \operatorname{End}(\mathbb{R}^2)$ and $b\in \mathbb{R}^2$.
Let $\{e_1,e_2\}$ be the canonical basis of $\mathbb{R}^2$ and let $A$ denote the matrix induced by $A$. By polar decomposition, we may assume $A=\Delta\cdot \Omega$, where $\Delta=\left(\smallmatrix d_1& 0\\ 0 & d_2 \endsmallmatrix\right)$
 is diagonal and $\Omega\in O(2)$. 
Let $(x_i,y_i)$ be the coordinate of $A_{i+1}-A_i$ and $(x_i^{o},y_i^{o})$ be the coordinate of $\Omega(A_{i+1}^{o}-A_i^{o})$, then
$$(x_i,y_i)=(x_i^{o}d_1,y_i^{o}d_2),\qquad (5.3)$$

So
$$(x_i^{o})^2d_1^2+(y_i^{o})^2d_2^2=x_i^2+y_i^2,\qquad (5.4)$$
Since $|A_iA_{i+1}|=1$, $|A^o_iA^o_{i+1}|=1$ and $\Omega$ is an isometry, we have  
$$x_i^2+y_i^2=1, \ \ (x_i^{o})^2+(y_i^{o})^2=1,\qquad (5.5)$$

By (5.4) and (5.5), we have $$(x_i^{o})^2d_1^2+(y_i^{o})^2d_2^2=1, \qquad (5.6)$$
By (5.5), plugging $(y_i^{o})^2=1-(x_i^{o})^2$ into (5.6), we have 
$$(x_i^{o})^2d_1^2+(1-(x_i^{o})^2)d_2^2=1,$$
that is,
$$(x_i^{o})^2(d_1^2-d_2^2)=1-d_2^2.$$
If $d_1^2\ne d_2^2$, then
$$(x_i^{o})^2=\frac{1-d_2^2}{d_1^2-d_2^2},\qquad (5.7)$$
so
$$(y_i^{o})^2=1-(x_i^{o})^2=\frac{1-d_1^2}{d_2^2-d_1^2}.$$
So
$$(x_i^{o},y_i^{o})=\left(\pm \sqrt{\frac{1-d_2^2}{d_1^2-d_2^2}},\pm \sqrt{\frac{1-d_1^2}{d_2^2-d_1^2}}\right) \qquad (5.8).$$
So  (5.8) implies that  $\{\Omega(A^o_iA^o_{i+1})\}_{i=1}^n$, the vertices of a regular $n$-gon centered at the origin $O$, occupy at most 4 positions in the plane,
but this is impossible, since, $n>4$. So $d_1^2=d_2^2$. Then by (5.7) we have 
$$d_1^2=d_2^2=1$$ and we get $\Delta\in O(2)$. So $A=\Delta\cdot \Omega\in O(2)$ and $\tau$ is an isometry of $\mathbb{R}^2$. So $\Gamma_n=\Gamma_n^o$.
\end{proof}
\begin{proof} [Proof of Theorem \ref{main3}]  The proof of Theorem \ref{main3} consists of  two cases: The case of $\alpha=2$ is based
on Theorem \ref{2}, and also need Lemma \ref{3}. The case of $\alpha <2$ is based on Theorem \ref{1}.

(1)  The case of  $\alpha=2$: (5.1) and (5.2) become
$$E_n^2(\Gamma_n)=\sum_{k=1}^{[\frac{n}{2}]}\mu_{n,k}{E_{n,k}^2}(\Gamma_n)\qquad (5.9)$$
and 
$$E_{n,k}^2(\Gamma_n)=\sum_{i=1}^{n}{|A_iA_{i+k}|^2}.\qquad (5.10)$$
By Theorem \ref{2} and (5.10), we have 
$$E_{n,k}^2(\Gamma_n)\le (\sin\frac{k\pi}{n}/\sin\frac{\pi}{n})^2\cdot n,\qquad(5.11)$$
By (5.9) and (5.11), we have 
$$E_n^2(\Gamma_n)\le\sum_{k=1}^{[\frac{n}{2}]}{\mu_{n,k}(\sin\frac{k\pi}{n}/\sin\frac{\pi}{n})^2}\cdot n.\qquad(5.12)$$
When $\Gamma_n=\Gamma_n^{o}$, by Theorem \ref{2}, the equal sign of $(5.11)$ holds. Then the equal sign of $(5.12)$ also holds. So $\Gamma_n^o$ is a maximal point of $E_n^2$, and the maximum equals 
 $$\sum_{k=1}^{[\frac{n}{2}]}\mu_{n,k}(\sin{\frac{k\pi}{n}/\sin\frac{\pi}{n})}^2\cdot n,$$
Now if $\Gamma_n$ is a maximal point of $E_n^2$, then
 $$E_n^2(\Gamma_n)=\sum_{k=1}^{[\frac{n}{2}]}\mu_{n,k}(\sin{\frac{k\pi}{n}/\sin\frac{\pi}{n})}^2\cdot n,$$
that is, the sign of equality holds in $(5.12)$, so the sign of equality holds in $(5.11)$, so by Theorem \ref{2}, $\Gamma_n=\tau(\Gamma_n^o)$, where $\tau$  is an affine transformation on $\mathbb{R}^2$ or $\tau$ is a limit of affine transformations on $\mathbb{R}^2$.
 By Lemma \ref{3}, $\Gamma_n=\Gamma_n^o$.

(2) The case of  $\alpha<2$: We define 
$$g_\alpha(x)=\begin{cases}
x^{\frac{\alpha}{2}}, \alpha>0, \\
\frac{1}{2}\ln(x), \alpha=0,\\
-x^{-\frac{\alpha}{2}}, \alpha<0
\end{cases}
$$
Then by (1.2) we have 
$$g_\alpha(x)=f_\alpha(\sqrt{x})=c_\alpha f_{\alpha/2}(x).$$
where $c_\alpha=1/2$ when $\alpha=0$ and $c_\alpha=1$ otherwise.
We rewrite   (5.2) as


$$E_{n,k}^\alpha(\Gamma_n)=\sum_{i=1}^{n}{g_\alpha(|A_iA_{i+k}|^2)}\qquad (5.13)$$

Since $\alpha<2$, by Lemma  \ref{concon}, 
$g_\alpha=f_{\alpha/2}$ is increasing and concave for any $x> 0$. 

Then by Theorem \ref{1} and (5.13), we have 
$$E_{n,k}^{\alpha}(\Gamma_n)\le n g_{\alpha} (\sin^2\frac{k\pi}{n} / \sin^2\frac{\pi}{n})\qquad(5.14)$$
By (5.14) and (5.1), we have 
$$E_n^{\alpha}(\Gamma_n)= \sum_{k=1}^{[\frac{n}{2}]}\mu_{n,k}{E_{n,k}^{\alpha}(\Gamma_n)\le \sum_{k=1}^{[\frac{n}{2}]}\mu_{n,k}g_{\alpha}(\sin^2\frac{k\pi}{n}/\sin^2\frac{\pi}{n})}\qquad(5.15)$$
When $\Gamma_n=\Gamma_n^{o}$ by Theorem \ref{1}, the equal sign of (5.14)  holds. Then the equal sign of (5.15) also holds. So $\Gamma_n^{o}$ is a maximal point of $E_n^\alpha$, and the maximum of $E_n^{\alpha}$ equals

$$\sum_{k=1}^{[\frac{n}{2}]}{\mu_{n,k}ng_\alpha(\sin^2\frac{k\pi}{n}/\sin^2\frac{\pi}{n}}).$$ If $\Gamma_n$ is a maximal point of $E_n^\alpha$, then $$E_n^{\alpha}(\Gamma_n)=\sum_{k=1}^{[\frac{n}{2}]}{\mu_{n,k}g_{\alpha}(\sin^2\frac{k\pi}{n}/\sin^2\frac{\pi}{n})}$$
Then $(5.15)$ holds, so the equal sign in (5.14) holds, so by Theorem \ref{1}, we have $\Gamma_n=\Gamma_n^{o}$.
\end{proof}




\begin{thebibliography}{HJSMM}

\bibitem[ACFGH]{ACFGH} A. Abrams, J. Cantarella, J. H. G. Fu, M. Ghomi, and R. Howard. {\it Circles minimize most knot energies.}   
  Topology, 42(2):381-394, 2003.





\bibitem[AGHPMP]{AGHPMP}  C. Audet, A. Guillou, P. Hansen, F. Messine, S. Perron,  {\it The small hexagon and heptagon with maximum sum of distances between vertices.} J. Global Optim. 49 (2011), no. 3, 467–480.


\bibitem[Ar]{Ar} M.A. Armstrong, {\it Basic Topology.} McGraw-Hill, 1979.

\bibitem[BG]{BG} J. Brauchart, P.  Grabner, {\it Distributing many points on spheres: 
minimal energy and designs}. J. Complexity 31 (2015), no. 3, 293–326.

\bibitem[CDR]{CDR} R. Connelly, E.D.  Demaine, G.  Rote, {\it Straightening polygonal arcs and convexifying polygonal cycles.} U.S.-Hungarian Workshops on Discrete Geometry and Convexity (Budapest, 1999/Auburn, AL, 2000). Discrete Comput. Geom. 30 (2003), no. 2, 205-239.

\bibitem[EHL]{EHL}  P. Exner, E. M. Harrell, and M. Loss, {\it Inequalities for means of chords, with application to isoperimetric problems}. Lett. Math. Phys., 75(3):225–233, 2006.

\bibitem[FHW]{FHW} M. H. Freedman, Z. X. He, and Z. H. Wang. {\it Mobius energy of knots and unknots.} Ann. of Math. (2), 139(1):1-50, 1994.

\bibitem[HLP]{HLP} G. H. Hardy, J. E. Littlewood,  G. Pólya, {\it Inequalities.} 2d ed. Cambridge, at the University Press, 1952.

\bibitem[HS]{HS} X. Hou, J.  Shao, {\it Spherical distribution of 5 points with maximal distance sum.} Discrete Comput. Geom. 46 (2011), no. 1, 156–174.

\bibitem[Ji]{Ji} M. Jiang,  {\it On the sum of distances along a circle.} Discrete Math. 308 (2008), no. 10, 2038–2045.

\bibitem[KuSa]{KuSa} A. B. J. Kuijlaars,  E. B. Saff, {\it Asymptotics for Minimal Discrete Energy on the Sphere.} Trans. Amer. Math. Soc. 350 (1998), 523–538.

\bibitem[KaSh]{KaSh} A. Katanforoush, M. Shahshahani, Mehrdad, 
{\it Distributing points on the sphere I.} 
Experiment. Math. 12 (2003), no. 2, 199–209. 

\bibitem[OH]{OH} J. O’Hara. {\it Family of energy functionals of knots.} Topology Appl., 48(2):147-161, 1992.

\bibitem[LP]{LP} G. Larcher, F. Pillichshammer, {\it  The sum of distances between vertices of a convex polygon with unit perimeter. }
Amer. Math. Monthly 115 (2008), no. 4, 350–355. 

\bibitem[Luko]{Luko} G. Lükő,
{\it On the mean length of the chords of a closed curve. }
Israel J. Math. 4 (1966), 23-32. 
 
\bibitem[Pi]{Pi} F. Pillichshammer, {\it On extremal point distributions in the Euclidean plane.} 
Acta Math. Hungar. 98 (2003), no. 4, 311–321.

\bibitem[PB]{PB} B. Plestenjak, V. Batagelj,  {\it Optimal arrangements of n points on a sphere and in a circle.} Proceedings (Preddvor), 83-87, Slov. Soc. Inform., Ljubljana (2001)

\bibitem[Sal]{Sal} Sallee, G. T. {\it Stretching chords of space curves.} Geometriae Dedicata 2 (1973), 311-315. 

\bibitem[Sch]{Sch} R.E. Schwartz,  {\it The five-electron case of Thomson's problem.} 
Exp. Math. 22 (2013), no. 2, 157–186.

\bibitem[Sm]{Sm} S. Smale,  {\it Mathematical Problems for the  Next Century.}  In Mathematics: Frontiers and Perspectives,
edited by Arnold, Atiyah, Lax, and Mazur.
Providence, RI: Amer. Math. Society, (2000).

\bibitem[St]{St} S. Steinerberger,  {\it On the optimal interpoint distance sum inequality.}  Arch. Math. (Basel) 97 (2011), no. 3, 289–298.

\bibitem[Th]{Th} J. J. Thomson, Philos. Mag. 7, 237 (1904).

\bibitem[To]{To} L. F. Tóth, {\it Über eine Punktverteilung auf der Kugel}, Acta Math. Hungar. 10 (1959) 13–19.

\bibitem[Wi]{Wi}  H. Witsenhausen, {\it On the maximum of the sum of squared distances under a diameter constraint.} Amer. Math. Monthly 81 (1974), 1100–1101.

\end{thebibliography}
\end{document}